\documentclass[11pt,english]{amsart}
\usepackage[T1]{fontenc}
\usepackage[latin9]{inputenc}
\usepackage{babel}
\usepackage{float}
\usepackage{amsthm}
\usepackage{amssymb}
\usepackage{graphicx}
\usepackage[letterpaper]{geometry}
\usepackage[pdfusetitle,
 bookmarks=true,bookmarksnumbered=false,bookmarksopen=false,
 breaklinks=false,pdfborder={0 0 1},backref=false,colorlinks=false]
 {hyperref}
\hypersetup{
 colorlinks=true,linkcolor=teal,citecolor=purple,linktocpage=true}

\makeatletter
\numberwithin{equation}{section}
\numberwithin{figure}{section}
\theoremstyle{plain}
\newtheorem{thm}{\protect\theoremname}[section]
\theoremstyle{plain}
\newtheorem{question}[thm]{\protect\questionname}
\theoremstyle{remark}
\newtheorem{rem}[thm]{\protect\remarkname}
\theoremstyle{plain}
\newtheorem{lem}[thm]{\protect\lemmaname}
\theoremstyle{plain}
\newtheorem{cor}[thm]{\protect\corollaryname}
\providecommand*{\code}[1]{\texttt{#1}}

\@ifundefined{date}{}{\date{}}
\usepackage{amsmath, tikz-cd, amssymb, placeins}
\usepackage{listings}
\DeclareRobustCommand*\cal{\@fontswitch\relax\mathcal}

\usetikzlibrary{calc}
\usetikzlibrary{decorations.pathmorphing}

\tikzset{curve/.style={settings={#1},to path={(\tikztostart)
    .. controls ($(\tikztostart)!\pv{pos}!(\tikztotarget)!\pv{height}!270:(\tikztotarget)$)
    and ($(\tikztostart)!1-\pv{pos}!(\tikztotarget)!\pv{height}!270:(\tikztotarget)$)
    .. (\tikztotarget)\tikztonodes}},
    settings/.code={\tikzset{quiver/.cd,#1}
        \def\pv##1{\pgfkeysvalueof{/tikz/quiver/##1}}},
    quiver/.cd,pos/.initial=0.35,height/.initial=0}

\tikzset{tail reversed/.code={\pgfsetarrowsstart{tikzcd to}}}
\tikzset{2tail/.code={\pgfsetarrowsstart{Implies[reversed]}}}
\tikzset{2tail reversed/.code={\pgfsetarrowsstart{Implies}}}
\tikzset{no body/.style={/tikz/dash pattern=on 0 off 1mm}}

\makeatother

\usepackage{listings}
\providecommand{\corollaryname}{Corollary}
\providecommand{\lemmaname}{Lemma}
\providecommand{\questionname}{Question}
\providecommand{\remarkname}{Remark}
\providecommand{\theoremname}{Theorem}

\begin{document}
\title{Khovanov homology can distinguish exotic Mazur manifolds}
\author{Gheehyun Nahm}
\thanks{The author was partially supported by the ILJU Academy and Culture
Foundation, the Simons collaboration \emph{New structures in low-dimensional
topology}, and a Princeton Centennial Fellowship.}
\address{Department of Mathematics, Princeton University, Princeton, New Jersey
08544, USA}
\email{gn4470@math.princeton.edu}
\begin{abstract}
In a recent breakthrough, Ren and Willis gave the first analysis-free
proof of the existence of exotic compact, orientable $4$-manifolds;
their main tool is the Khovanov skein lasagna module defined by Morrison,
Walker, and Wedrich. In this paper, we introduce a new, simple way
of using Khovanov homology to distinguish certain exotic compact,
orientable $4$-manifolds; our new method does not depend on the skein
lasagna module. As an application, we give the first analysis-free
proof of the existence of exotic Mazur manifolds, i.e.\ compact,
contractible $4$-manifolds that have handle decompositions with a
single $1$-handle and a single $2$-handle.
\end{abstract}

\maketitle

\section{Introduction}

In a recent breakthrough \cite{2402.10452}, Ren and Willis gave the
first analysis-free proof of the existence of exotic compact, orientable
$4$-manifolds; their main tool is the Khovanov skein lasagna module
defined by Morrison, Walker, and Wedrich \cite{MWW}. In this paper,
we first give a simple, skein lasagna-free proof (Corollary~\ref{cor:For-,-the})
of the existence of exotic compact, orientable $4$-manifolds. This
proof only uses that the Khovanov map for link cobordisms $S\subset[0,1]\times\mathbb{R}^{3}$
only depends, up to sign, on the isotopy class of $S$ rel $\partial$
\cite{MR1740682,Ja}, and Hayden and Sundberg's proof that this Khovanov
map distinguishes certain exotic disks in $D^{4}$ \cite{MR4726569}.

We then extend our method and give the first analysis-free proof (Theorem~\ref{thm:Khovanov-homology-can})
of the existence of exotic Mazur manifolds, i.e.\ compact, contractible
$4$-manifolds that have handle decompositions with a single $1$-handle
and a single $2$-handle. This proof is also skein lasagna-free and
uses little machinery; the key technical input is Ren's partial computation
of the Khovanov homology of torus links $T(n,n)$ and maps between
them \cite{MR4843752}. Note that the first pair of exotic compact,
contractible $4$-manifolds was obtained by Akbulut and Ruberman \cite{MR3471934},
and examples of exotic pairs of Mazur manifolds have previously been
discovered as well \cite{MR4053349,MR4317407}.

\begin{thm}
\label{thm:Khovanov-homology-can}For every integer $k\ge1$, Khovanov
homology can distinguish the exotic pair of Mazur manifolds of Figure~\ref{fig:mazur}.
\end{thm}

Theorem~\ref{thm:Khovanov-homology-can} is proved in Section~\ref{sec:Exotic-compact,-contractible}.
These $4$-manifolds are the exteriors of the disks in $\overline{\mathbb{CP}^{2}}\backslash{\rm int}D^{4}$
given by blowing up Hayden and Sundberg's \cite{MR4726569} exotic
asymmetric disks in $D^{4}$ (Figure~\ref{fig:asymdisk}) at a point
on the disks: see Figure~\ref{fig:handlecalculus} for alternative
handle diagrams of these manifolds. We thank Kyle Hayden for pointing
out that they are Mazur manifolds.

Ren and Willis \cite[Section 6.11]{2402.10452} defined the Khovanov
cobordism map for oriented surfaces in $\mathbb{CP}^{2}\backslash{\rm int}D^{4}$,
and showed, using skein lasagna modules, that it only depends on the
isotopy class of the surface rel $\partial$. A key ingredient of
our proof of Theorem~\ref{thm:Khovanov-homology-can} is Lemma~\ref{lem:tqft-diff},
which says that this cobordism map is invariant under diffeomorphisms
of $\mathbb{CP}^{2}\backslash{\rm int}D^{4}$ that fix the boundary
pointwise. Although we were motivated by skein lasagna modules, we
also present a proof in Subsection~\ref{subsec:A-direct-argument}
that avoids the theory of skein lasagna modules and instead uses \cite{MR4843752}
directly, with the aim of making the argument more transparent.

\begin{figure}[H]
\begin{centering}
\includegraphics[scale=0.7]{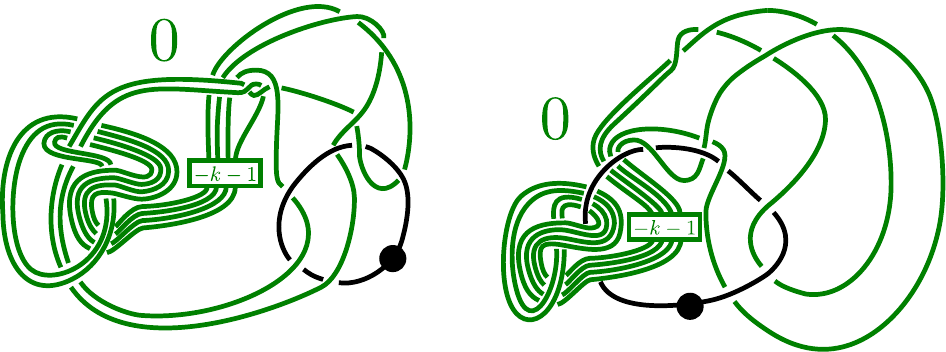}
\par\end{centering}
\caption{\label{fig:mazur}Exotic pairs of Mazur manifolds ($k\in\mathbb{Z}$,
$k\ge1$).}
\end{figure}

Theorem~\ref{thm:Khovanov-homology-can} still leaves the following
questions open.
\begin{question}
Can Khovanov homology distinguish exotic closed, oriented $4$-manifolds?
\end{question}

\begin{question}
Can Khovanov homology distinguish exotic closed, oriented, simply
connected $4$-manifolds?
\end{question}

\subsection*{Conventions}

Homeomorphisms, diffeomorphisms, and isotopies are \emph{rel $\partial$}
if they fix the boundary pointwise.

If $M$ is an oriented manifold, then $-M$ is $M$ equipped with
the opposite orientation.

The oriented boundary of the standard $4$-ball $D^{4}$ is $S^{3}$,
and if $\Sigma\subset D^{4}$ is a properly embedded surface with
boundary $L\subset S^{3}$, then the corresponding cobordism maps
on Khovanov homology are 
\[
Kh(\Sigma):Kh(\emptyset)\to Kh(L),\ Kh(m(L))\to Kh(\emptyset),
\]
where $m(L)$ is the \emph{mirror of $L$}. Similarly, if we denote
the \emph{mirror of $\Sigma$} as $m(\Sigma)$, then it induces
\[
Kh(m(\Sigma)):Kh(\emptyset)\to Kh(m(L)),\ Kh(L)\to Kh(\emptyset).
\]

\begin{rem}
\label{rem:Our-orientation-conventions,}Our orientation conventions,
in effect, are the same as \cite{MWW} (see \cite[Example 5.6]{MWW})
and \cite[Section 6.11]{2402.10452}, but are the opposite of \cite{MR4726569}.
\end{rem}

\subsection*{Acknowledgements}

We thank Peter Ozsv\'{a}th for his continuous support and helpful
discussions. The main observation of this work was made while rereading
the author's Part III essay during a visit to Duke University. We
thank Duke University for their hospitality and Jacob Rasmussen for
advising the author's Part III essay \emph{Khovanov Homology and Embedded
Surfaces}. We thank Nathan Dunfield for help with SnapPy, Kyle Hayden,
Adam Levine, Lisa Piccirillo, Misha Schmalian, Alison Tatsuoka, and
Qiuyu Ren for helpful discussions, Kyle Hayden and Robert Lipshitz
for helpful comments on earlier drafts, and Selman Akbulut for pointing
out a reference.

\section{\label{sec:Exotic-compact,-orientable}Exotic compact, orientable
$4$-manifolds from exotic disks}

In \cite[Example 3.6]{MR4726569}, Hayden and Sundberg define knots
$J_{k}\subset S^{3}$ for $k\in\mathbb{Z}$ and a pair of ribbon disks
$\Sigma_{k},\Sigma_{k}'\subset D^{4}$ that $J_{k}$ bounds (denoted
as $J_{m},\Sigma_{m},\Sigma_{m}'$ respectively in \cite{MR4726569}).
With respect to the radial height function of $D^{4}$, these ribbon
disks have two index $1$ critical points and three index $0$ critical
points. The knot $J_{k}$ and the two bands that correspond to the
index $1$ critical points are drawn in Figure~\ref{fig:asymdisk}.
\begin{figure}[H]
\begin{centering}
\includegraphics[scale=0.5]{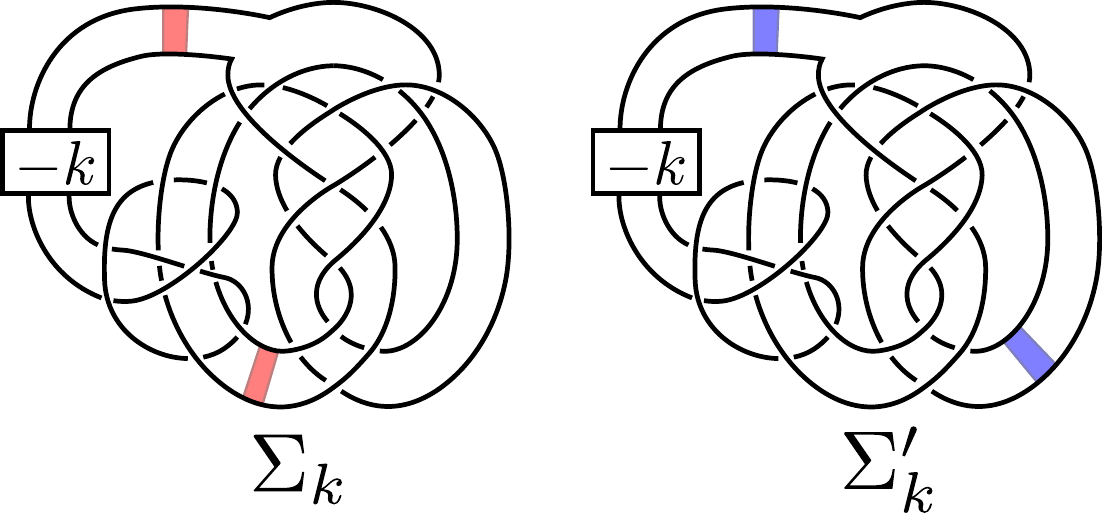}
\par\end{centering}
\caption{\label{fig:asymdisk}\cite[Figure 6]{MR4726569} The exotic slice
disks $\Sigma_{k},\Sigma_{k}'$ of $J_{k}$}
\end{figure}

\begin{thm}[{\cite[Example 3.6]{MR4726569} and proof of \cite[Theorem 1.1]{MR4726569}}]
\label{thm:HS1}For every integer $k\ge1$, the ribbon disks $\Sigma_{k},\Sigma_{k}'\subset D^{4}$
of Figure~\ref{fig:asymdisk} are topologically ambiently isotopic
rel $\partial$ but not smoothly ambiently isotopic rel $\partial$.
In fact, there exists an element $\phi\in Kh(J_{k})$ such that $Kh(m(\Sigma_{k}))(\phi)=\pm1$
and $Kh(m(\Sigma_{k}'))(\phi)=0$, where 
\[
Kh(m(\Sigma_{k})),Kh(m(\Sigma_{k}')):Kh(J_{k})\to\mathbb{Z}
\]
are the cobordism maps on Khovanov homology over $\mathbb{Z}$, where
$m(\Sigma_{k})$ (resp.\ $m(\Sigma_{k}')$) means the mirror of $\Sigma_{k}$
(resp.\ $\Sigma_{k}'$).
\end{thm}

Recall that the link cobordism maps on Khovanov homology are well-defined
up to sign, and only depend on the isotopy class rel $\partial$ of
the surface \cite{MR1740682,Ja,MWW}.
\begin{rem}
For Theorem~\ref{thm:HS1}, we do not need that Khovanov homology
satisfies the sweep-around move \cite{MWW}; we only need that the
Khovanov map for link cobordisms $S\subset[0,1]\times\mathbb{R}^{3}$
only depends, up to sign, on the isotopy class of $S$ rel $\partial$
\cite{MR1740682,Ja}.

Indeed, let $\Sigma,\Sigma'\subset D^{4}$ be two disks such that
$\partial\Sigma=\partial\Sigma'$. Let $x\in\partial D^{4}$ be away
from $\partial\Sigma=\partial\Sigma'$, and identify a small closed
neighborhood $\overline{N}(x)\subset D^{4}$ of $x$ with $[0,1]\times D^{3}$,
where $\overline{N}(x)\cap\partial D^{4}=0\times D^{3}$ and $x=(0,0)\in0\times D^{3}$.
If there exists an isotopy rel $\partial$ between two disks $\Sigma,\Sigma'\subset D^{4}$,
then the isotopy does not intersect a sufficiently small neighborhood
$[0,\varepsilon]\times\varepsilon D^{3}\subset\overline{N}(x)$ of
$x$. Identify $D^{4}\backslash([0,\varepsilon)\times\varepsilon D^{3})\cong[0,1]\times\mathbb{R}^{3}$;
then $\Sigma,\Sigma'$ are link cobordisms in $[0,1]\times\mathbb{R}^{3}$
which are isotopic rel $\partial$ in $[0,1]\times\mathbb{R}^{3}$.
\end{rem}

\begin{figure}[H]
\begin{centering}
\includegraphics[scale=0.5]{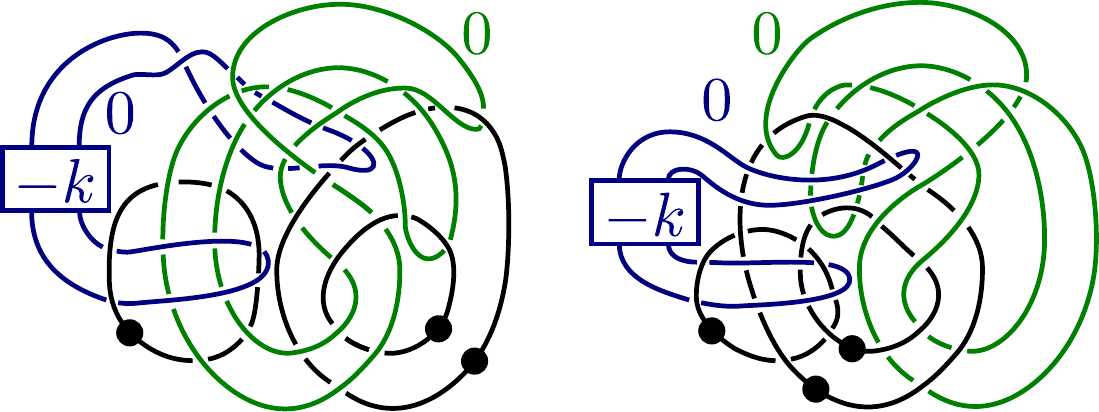}
\par\end{centering}
\caption{\label{fig:disk-exterior}Handle diagrams of the exotic pairs of disk
exteriors $D^{4}\backslash N(\Sigma_{k})$ and $D^{4}\backslash N(\Sigma_{k}')$}
\end{figure}

We deduce from Theorem~\ref{thm:HS1} (Corollary~\ref{cor:For-,-the})
that for every $k\ge1$, the compact, orientable $4$-manifolds of
Figure~\ref{fig:disk-exterior} are exotic. This follows from (1)
that they are the disk exteriors $D^{4}\backslash N(\Sigma_{k})$
and $D^{4}\backslash N(\Sigma_{k}')$ (see \cite[Section 6.2]{MR1707327})
where $N$ means open tubular neighborhood, (2) the following well
known lemma (Lemma~\ref{lem:wellknown}), and (3) that the mapping
class group of the boundary $S_{0}^{3}(J_{k})$ (the $0$-surgery
of $S^{3}$ along the knot $J_{k}$) of $D^{4}\backslash N(\Sigma_{k})$
and $D^{4}\backslash N(\Sigma_{k}')$ are trivial.
\begin{lem}[{\cite[Lemma 4.7]{MWW}, \cite[Lemma 4.7]{MR4504654}, \cite{334386}}]
\label{lem:wellknown}Let $\Sigma,\Sigma'\subset D^{4}$ be properly
embedded surfaces with the same boundary. If the pairs $(D^{4},\Sigma)$
and $(D^{4},\Sigma')$ are diffeomorphic rel $\partial$, then $\Sigma$
and $\Sigma'$ are smoothly ambiently isotopic rel $\partial$.
\end{lem}

\begin{proof}
Assume that there exists a diffeomorphism $f:(D^{4},\Sigma)\to(D^{4},\Sigma')$
that is the identity on $\partial D^{4}$. After modifying $f,\Sigma,\Sigma'$
if necessary, we can further assume that $\Sigma,\Sigma'$ avoid a
neighborhood of the origin, say $\frac{1}{4}D^{4}$, and that $f$
is the identity on a neighborhood of the boundary, say $D^{4}\backslash\frac{1}{3}D^{4}$.

Consider a radial smooth ambient isotopy rel $\partial$ $(\varphi_{t}):D^{4}\times I\to D^{4}$
that ``dilates'' $\frac{1}{4}D^{4}$ to $\frac{1}{2}D^{4}$. Then
$(\varphi_{t})$ is a smooth ambient isotopy between $\Sigma$ and
$\varphi_{1}(\Sigma)$, and $(f\circ\varphi_{t}\circ f^{-1})$ is
a smooth ambient isotopy between $\Sigma'$ and $f\circ\varphi_{1}\circ f^{-1}(\Sigma')=\varphi_{1}(\Sigma)$.
Hence $\Sigma$ and $\Sigma'$ are smoothly ambiently isotopic.
\end{proof}

\begin{cor}
\label{cor:For-,-the}For every integer $k\ge1$, the compact, orientable
$4$-manifolds $D^{4}\backslash N(\Sigma_{k})$ and $D^{4}\backslash N(\Sigma_{k}')$
are homeomorphic but not diffeomorphic.
\end{cor}

\begin{proof}
First, $D^{4}\backslash N(\Sigma_{k})$ and $D^{4}\backslash N(\Sigma_{k}')$
are homeomorphic rel $\partial$ since the disks $\Sigma_{k}$ and
$\Sigma_{k}'$ are topologically ambiently isotopic rel $\partial$
by Theorem~\ref{thm:HS1} (see, for instance, \cite[Lemma 2.5]{CP}).\footnote{In fact, that the disks are topologically ambiently isotopic rel $\partial$
is proved by first proving that their exteriors are homeomorphic rel
$\partial$: Hayden and Sundberg \cite{MR4726569} use \cite{CP}
to show that the disks are topologically ambiently isotopic, and Conway
and Powell \cite{CP} first show that the disk exteriors are homeomorphic
rel $\partial$ (see the last paragraph of the proof of Theorem~1.4).}

Let us show that $D^{4}\backslash N(\Sigma_{k})$ and $D^{4}\backslash N(\Sigma_{k}')$
are not diffeomorphic. First, Theorem~\ref{thm:HS1} and Lemma~\ref{lem:wellknown}
imply that $(D^{4},\Sigma_{k})$ and $(D^{4},\Sigma_{k}')$ are not
diffeomorphic rel $\partial D^{4}$, and so $D^{4}\backslash N(\Sigma_{k})$
and $D^{4}\backslash N(\Sigma_{k}')$ are not diffeomorphic rel $\partial$.
Hence, we are left to show that the mapping class group of the boundary
$S_{0}^{3}(J_{k})$ is trivial for all $k\ge1$. (To the cautious
reader: the below indeed also checks that there are no orientation
reversing self-diffeomorphisms of $S_{0}^{3}(J_{k})$.)

To show that ${\rm MCG}(S_{0}^{3}(J_{k}))$ is trivial, we use SnapPy
\cite{SnapPy} inside Sage \cite{sagemath}. Let us first check it
for $k=1$; we handle the general case in Appendix~\ref{sec:Computation-of-the}
similarly to \cite[Proposition A.3 (a)]{MR4726569}. First, input
$J_{1}$ \cite[\code{j1.lnk}]{ancillaryfiles} as a \code{snappy.ManifoldHP}
object called \code{L}. The following code verifies that $S_{0}^{3}(J_{1})$
is hyperbolic and finds all geodesics with length $\le1$.
\begin{lstlisting}[language=Python,basicstyle={\ttfamily},breaklines=true,showstringspaces=false,columns=fullflexible,keepspaces=true]
L.dehn_fill((0,1))
L.verify_hyperbolicity()   # True
L.length_spectrum_alt(max_len=1, verified=True, bits_prec=1000)
\end{lstlisting}
SnapPy outputs a unique geodesic with real length $0.92213444882961\cdots$
and word \code{cJQpeID} (SnapPy may output a different word). Hence,
any isometry of $S_{0}^{3}(J_{1})$ must fix this geodesic. Now, use
the following code to check that the isometry group of the cusped
hyperbolic manifold given by drilling out that geodesic is trivial.
\begin{lstlisting}[language=Python,basicstyle={\ttfamily},breaklines=true,showstringspaces=false,columns=fullflexible,keepspaces=true]
R = L.drill_word('cJQpeID',verified=True).filled_triangulation().canonical_retriangulation(verified=True)
len(R.isomorphisms_to(R)) # 1
\end{lstlisting}
\end{proof}

\section{\label{sec:Khovanov-homology-and}Khovanov homology and oriented
surfaces in \texorpdfstring{$\mathbb{CP}^{2}\backslash{\rm int}D^{4}$}{CP2{\textbackslash}int(D4)}}

For properly embedded, oriented surfaces $S$ in $k\mathbb{CP}^{2}\backslash({\rm int}D^{4}\sqcup{\rm int}D^{4})$,
Ren and Willis define \cite[Section 6.11]{2402.10452} a map on Khovanov
homology, which is well-defined up to sign and only depends on the
isotopy class rel $\partial$ of $S$. In this section, we recall
the definition for the special case where $S$ is in $D^{4}\#\mathbb{CP}^{2}=:(\mathbb{CP}^{2})^{\circ}$,\footnote{We write $D^{4}\#\mathbb{CP}^{2}$ to make explicit that its oriented
boundary is $S^{3}$ and to avoid confusion in Section~\ref{sec:Exotic-compact,-contractible},
where we blow up surfaces in $D^{4}$.} sketch a direct (skein lasagna free) proof that it only depends on
the isotopy class rel $\partial$, and show that the map is invariant
under diffeomorphisms of $(\mathbb{CP}^{2})^{\circ}$ rel $\partial$.
The readers who are (resp.\ are not) familiar with skein lasagna
modules may skip Subsection~\ref{subsec:A-direct-argument} (resp.\ Subsection~\ref{subsec:A-skein-lasagna}).

\subsection{\label{subsec:A-direct-argument}A direct argument}

Let us first recall the definition of the cobordism map on Khovanov
homology for oriented surfaces $S$ properly embedded in $(\mathbb{CP}^{2})^{\circ}$.
For simplicity, we work over $\mathbb{Z}$. Further assume that $S$
intersects the core $\mathbb{CP}^{1}$ transversely, positively at
$p$ points and negatively at $q$ points. Let $N(\mathbb{CP}^{1})$
be a small open tubular neighborhood of $\mathbb{CP}^{1}$. Then,
$S\cap\partial N(\mathbb{CP}^{1})\subset\partial N(\mathbb{CP}^{1})\cong S^{3}$
is the \emph{negative} $(p+q,p+q)$ torus link, where $p$ of the
strands are oriented oppositely to the other $q$ strands. We denote
the \emph{mirror} of such oriented links as $T(p+q,p+q)_{p,q}$. The
cobordism map $Kh_{\mathbb{CP}^{2}}(S)$\footnote{We denote it as $Kh_{\mathbb{CP}^{2}}(S)$ for notational clarity
and to emphasize that $S$ is in $(\mathbb{CP}^{2})^{\circ}$.} is defined, up to sign, as the composition 
\begin{equation}
Kh_{\mathbb{CP}^{2}}(S):Kh(m(\partial S))\to Kh(T(p+q,p+q)_{p,q})\to\mathbb{Z},\label{eq:kh-cp2}
\end{equation}
where we define the two maps as follows.

The oriented boundary of $(\mathbb{CP}^{2})^{\circ}\backslash N(\mathbb{CP}^{1})$
is $\partial D^{4}\sqcup(-\partial N(\mathbb{CP}^{1}))$, and the
link $S\cap(-\partial N(\mathbb{CP}^{1}))\subset-\partial N(\mathbb{CP}^{1})\cong S^{3}$
is $T(p+q,p+q)_{p,q}$. Hence, we have a link cobordism
\[
((\mathbb{CP}^{2})^{\circ}\backslash N(\mathbb{CP}^{1}),S\cap((\mathbb{CP}^{2})^{\circ}\backslash N(\mathbb{CP}^{1}))):(S^{3},m(\partial S))\to(S^{3},T(p+q,p+q)_{p,q})
\]
and $(\mathbb{CP}^{2})^{\circ}\backslash N(\mathbb{CP}^{1})\cong[0,1]\times S^{3}$.
The first map of (\ref{eq:kh-cp2}) is the induced cobordism map.

Let ${\rm gr}_{q}(p,q)$ be the quantum filtration degree of the Lee
generator of $T(p+q,p+q)_{p,q}$ over $\mathbb{Q}$. Then, Ren shows
\cite[Corollary 2.2]{MR4843752} that 
\[
Kh^{0,{\rm gr}_{q}(p,q)}(T(p+q,p+q)_{p,q})\cong\mathbb{Z}.
\]
The second map of (\ref{eq:kh-cp2}) is projection onto the $(0,{\rm gr}_{q}(p,q))$
grading summand.
\begin{rem}
Ren \cite[Theorem 1.1]{MR4843752} also shows that ${\rm gr}_{q}(p,q)=(p-q)^{2}-2{\rm max}(p,q)$,
and so the cobordism map $Kh_{\mathbb{CP}^{2}}(S)$ has $(h,q)$-bidegree
$(0,\chi(S)-\alpha^{2}+|\alpha|)$, where $\mathbb{Z}$ is supported
in bidegree $(0,0)$ and $[S]\in H_{2}((\mathbb{CP}^{2})^{\circ})$
is $\alpha=p-q$ times a generator.
\end{rem}

\begin{rem}
\label{rem:blowdown}From this description, we can see that if $S$
intersects the core $\mathbb{CP}^{1}$ of $\mathbb{CP}^{2}$ exactly
once, then the cobordism map $Kh_{\mathbb{CP}^{2}}(S)$ is the cobordism
map of the surface in $D^{4}$ given by \emph{blowing down} $\mathbb{CP}^{1}$.
In other words: let $\Sigma\subset D^{4}$ be the surface obtained
from $S$ by replacing $\overline{N}(\mathbb{CP}^{1})$ with $D^{4}$
and capping the unknot $S\cap\partial N(\mathbb{CP}^{1})$ with a
boundary parallel disk in this $D^{4}$. Then $Kh_{\mathbb{CP}^{2}}(S)$
is the cobordism map $Kh(\Sigma):Kh(m(\partial S))\to\mathbb{Z}$
induced by $\Sigma$.
\end{rem}

Let us sketch a direct proof that $Kh_{\mathbb{CP}^{2}}(S)$ only
depends on the isotopy class rel $\partial$ of $S$ up to sign. This
is similar to the proof that $\Phi^{-1}$ is well-defined in the proof
of \cite[Theorem 1.1]{MR4445546}. Since the Khovanov cobordism map
only depends, up to sign, on the isotopy class rel $\partial$ for
surfaces in $[0,1]\times S^{3}$, we only have to check that $Kh_{\mathbb{CP}^{2}}(S)$
is invariant, up to sign, under isotopies supported near $\mathbb{CP}^{1}$,
and hence only for the following two kinds of isotopies: (1) those
that are induced by moving the intersections with $\mathbb{CP}^{1}$
around, and (2) those that introduce or remove two intersections with
$\mathbb{CP}^{1}$, of opposite sign.

The first case corresponds to checking that isotoping the strands
of $T(p+q,p+q)_{p,q}$ induces an isomorphism on $Kh^{0,{\rm gr}_{q}(p,q)}\cong\mathbb{Z}$,
which is clear.

For the second case, consider the composition of two saddle cobordisms
\begin{equation}
T(p+q,p+q)_{p,q}\to T(p+q,p+q)_{p,q}\sqcup U\to T(p+q+2,p+q+2)_{p+1,q+1}\label{eq:two-saddle}
\end{equation}
where $U$ is an unlinked unknot. The second case corresponds to checking
that the induced map
\begin{equation}
Kh^{0,{\rm gr}_{q}(p,q)}(T(p+q,p+q)_{p,q})\to Kh^{0,{\rm gr}_{q}(p+1,q+1)}(T(p+q+2,p+q+2)_{p+1,q+1})\label{eq:k1iso}
\end{equation}
is an isomorphism. This follows from \cite[Theorem 2.1 (1)]{MR4843752}:
the first statement of \cite[Theorem 2.1 (1)]{MR4843752} is that
$Kh^{0,i}(T(p+q,p+q)_{p,q})=0$ for $i<{\rm gr}_{q}(p,q)$. Hence,
\begin{align*}
Kh^{0,{\rm gr}_{q}(p,q)}(T(p+q,p+q)_{p,q}) & =Kh^{0,{\rm gr}_{q}(p,q)}(T(p+q,p+q)_{p,q})\oplus Kh^{0,{\rm gr}_{q}(p,q)-2}(T(p+q,p+q)_{p,q})\\
 & \cong Kh^{0,{\rm gr}_{q}(p,q)-1}(T(p+q,p+q)_{p,q}\sqcup U),
\end{align*}
and the first saddle cobordism of (\ref{eq:two-saddle}) induces this
isomorphism. The second statement of \cite[Theorem 2.1 (1)]{MR4843752}
is that the second saddle cobordism of (\ref{eq:two-saddle}) induces
an isomorphism
\[
Kh^{0,{\rm gr}_{q}(p,q)-1}(T(p+q,p+q)_{p,q}\sqcup U)\to Kh^{0,{\rm gr}_{q}(p+1,q+1)}(T(p+q+2,p+q+2)_{p+1,q+1}).
\]
Combining the above two, we obtain that (\ref{eq:k1iso}) is an isomorphism.

Now, we show that $Kh_{\mathbb{CP}^{2}}(S)$ is invariant under diffeomorphisms
of $(\mathbb{CP}^{2})^{\circ}$ rel $\partial$.
\begin{lem}
\label{lem:tqft-diff}Let $S,S'$ be properly embedded, oriented surfaces
in $(\mathbb{CP}^{2})^{\circ}:=D^{4}\#\mathbb{CP}^{2}$ such that
$[S]=[S']\in H_{2}((\mathbb{CP}^{2})^{\circ};\mathbb{Z})$. If there
exists a diffeomorphism 
\[
((\mathbb{CP}^{2})^{\circ},S)\cong((\mathbb{CP}^{2})^{\circ},S')\ {\rm rel}\ \partial,
\]
then the induced maps 
\[
Kh_{\mathbb{CP}^{2}}(S),Kh_{\mathbb{CP}^{2}}(S'):Kh(m(\partial S))\to\mathbb{Z}
\]
agree up to sign.
\end{lem}

\begin{proof}
We proceed similarly to the proof of Lemma~\ref{lem:wellknown}.
View $(\mathbb{CP}^{2})^{\circ}$ as a unit disk bundle over the core
$\mathbb{CP}^{1}$. For $r\in(0,1]$, denote $N_{r}\subset(\mathbb{CP}^{2})^{\circ}$
as the open disk subbundle over $\mathbb{CP}^{1}$ with radius $r$.

Assume that there exists a diffeomorphism rel $\partial$ $f:((\mathbb{CP}^{2})^{\circ},S)\to((\mathbb{CP}^{2})^{\circ},S')$.
After isotoping $f,S,S'$ if necessary, assume that $S\cap N_{1/4}$
and $S'\cap N_{1/4}$ are a finite number of fiber disks, and that
$f$ is the identity on $(\mathbb{CP}^{2})^{\circ}\backslash N_{1/3}$.
Consider a fiberwise smooth ambient isotopy rel $\partial$ $(\varphi_{t}):(\mathbb{CP}^{2})^{\circ}\times I\to(\mathbb{CP}^{2})^{\circ}$
that ``dilates'' $N_{1/4}$ to $N_{1/2}$. Then $(\varphi_{t})$
is a smooth ambient isotopy between $S$ and $\varphi_{1}(S)$, and
$(f\circ\varphi_{t}\circ f^{-1})$ is a smooth ambient isotopy between
$S'$ and $f\circ\varphi_{1}(S)$.

We claim that the cobordism maps $Kh_{\mathbb{CP}^{2}}(\varphi_{1}(S))$
and $Kh_{\mathbb{CP}^{2}}(f\circ\varphi_{1}(S))$ agree. Note that
the surfaces $\varphi_{1}(S)$ and $f\circ\varphi_{1}(S)$ agree on
$(\mathbb{CP}^{2})^{\circ}\backslash N_{1/3}$. Let $\varphi_{1}(S)$
intersect $\mathbb{CP}^{1}$ at $p$ points positively and $q$ points
negatively. Since $\varphi_{1}(S)\cap N_{1/3}$ is some number of
fiber disks, $Kh_{\mathbb{CP}^{2}}(\varphi_{1}(S))$ is induced by
the Khovanov map given by the cobordism
\[
((\mathbb{CP}^{2})^{\circ}\backslash N_{1/3},\varphi_{1}(S)\cap((\mathbb{CP}^{2})^{\circ}\backslash N_{1/3})):(S^{3},m(\partial S))\to(S^{3},T(p+q,p+q)_{p,q}).
\]

By modifying $f\circ\varphi_{1}(S)$ near $\mathbb{CP}^{1}$ if necessary,
assume that there exists an $\varepsilon>0$ such that $(f\circ\varphi_{1}(S))\cap N_{\varepsilon}$
is $(p+q+2\ell)$ many fiber disks. Note that $\ell\ge0$. Then, we
are left to show that 
\begin{equation}
(\overline{N_{1/3}}\backslash N_{\varepsilon},S\cap(\overline{N_{1/3}}\backslash N_{\varepsilon})):(S^{3},T(p+q,p+q)_{p,q})\to(S^{3},T(p+q+2\ell,p+q+2\ell)_{p+\ell,q+\ell})\label{eq:cobord}
\end{equation}
induces an isomorphism
\begin{equation}
Kh^{0,{\rm gr}_{q}(p,q)}(T(p+q,p+q)_{p,q})\to Kh^{0,{\rm gr}_{q}(p+\ell,q+\ell)}(T(p+q+2\ell,p+q+2\ell)_{p+\ell,q+\ell}).\label{eq:kiso}
\end{equation}

In fact, we claim that if $A$ is a surface in $[0,1]\times S^{3}$
with the same domain and codomain as (\ref{eq:cobord}) and is topologically
the disjoint union of $(p+q)$ many annuli with $2\ell$ disks removed,
then the induced cobordism map (\ref{eq:kiso}) is an isomorphism.
Note that the surface in (\ref{eq:cobord}) satisfies this condition.
To show the claim, we use cobordism maps on the Lee spectral sequence
\cite{MR2173845,rasmussen2005khovanovsinvariantclosedsurfaces,Ra}
over $\mathbb{Q}$. Let $\mathfrak{s}_{p,q}$ (resp.\ $\mathfrak{s}_{p+\ell,q+\ell}$)
be the Lee generator of $T(p+q,p+q)_{p,q}$ (resp.\ $T(p+q+2\ell,p+q+2\ell)_{p+\ell,q+\ell}$),
and let us also denote $\mathfrak{s}_{p,q}$ (resp.\ $\mathfrak{s}_{p+\ell,q+\ell}$)
as its image in the $E_{\infty}$ page of the Lee spectral sequence.
Then, for any such $A$, the induced cobordism map on the $E_{\infty}$
page of the Lee spectral sequence over $\mathbb{Q}$ maps $\mathfrak{s}_{p,q}$
to $\pm C\mathfrak{s}_{p+\ell,q+\ell}$, where $C$ only depends on
$p,q,\ell$. Since the domain and the codomain of (\ref{eq:k1iso})
are $\mathbb{Z}$ and these gradings are where the Lee generators
$\mathfrak{s}_{p,q}$ and $\mathfrak{s}_{p+\ell,q+\ell}$ (respectively)
live over $\mathbb{Q}$, the domain and codomain of (\ref{eq:k1iso})
survive to the $E_{\infty}$ page of the Lee spectral sequence over
$\mathbb{Q}$. Hence, any such $A$ and $A'$ induce the same map
(\ref{eq:kiso}) up to sign. To show that (\ref{eq:kiso}) is an isomorphism
over $\mathbb{Z}$, we only have to check it for one such $A$: one
choice is the composition of $\ell$ many of (\ref{eq:k1iso}).
\end{proof}

\subsection{\label{subsec:A-skein-lasagna}A skein lasagna argument}

In this subsection, we present a skein lasagna proof of Lemma~\ref{lem:tqft-diff}.
For simplicity, we work over $\mathbb{Q}$ in this subsection, unless
specified otherwise.

Ren and Willis \cite[Section 6.11]{2402.10452} give a description
of the cobordism map 
\[
Kh_{\mathbb{CP}^{2}}(S):Kh(m(\partial S))\to\mathbb{Q}
\]
for oriented surfaces $S\subset(\mathbb{CP}^{2})^{\circ}$, in terms
of the Khovanov-Rozansky $\mathfrak{gl}_{2}$ skein lasagna module.
Note that this latter description is phrased in terms of the $\mathfrak{gl}_{2}$
homology $KhR_{2}$ instead of $Kh$, which can be thought as a renormalization
of $Kh$ of the \emph{mirror} of the link \cite[Equation (8)]{2402.10452};
hence the cobordism map is defined for oriented surfaces in $D^{4}\#\overline{\mathbb{CP}^{2}}$.
For a quick introduction to the $\mathfrak{gl}_{2}$ skein lasagna
modules, we refer the readers to \cite[Sections 2.1-2.4]{2402.10452};
see the original paper \cite{MWW} for the general setting.

Let $S$ be an oriented surface in $D^{4}\#\overline{\mathbb{CP}^{2}}$.
Let us recall the skein lasagna description of the cobordism map 
\begin{equation}
KhR_{2,\overline{\mathbb{CP}^{2}}}(S):KhR_{2}(m(\partial S))\to\mathbb{Q}.\label{eq:khrcobord}
\end{equation}
Let $\alpha:=[S]\in H_{2}(D^{4}\#\overline{\mathbb{CP}^{2}})=H_{2}(\overline{\mathbb{CP}^{2}}).$
Ren and Willis show that ${\cal S}_{0,0,|\alpha|}^{2}(\overline{\mathbb{CP}^{2}};\alpha;\mathbb{Z})\cong\mathbb{Z}$,
and define the \emph{canonical dual lasagna generator} $\theta_{\alpha}\in({\cal S}_{0}^{2}(\overline{\mathbb{CP}^{2}};\alpha;\mathbb{Q}))_{0,-|\alpha|}^{\ast}$
up to sign, as the dual element that maps a generator of ${\cal S}_{0,0,|\alpha|}^{2}(\overline{\mathbb{CP}^{2}};\alpha;\mathbb{Z})\cong\mathbb{Z}$
to $1$. Consider the lasagna filling with no input balls given by
the surface $S$. This defines an element in ${\cal S}_{0}^{2}(D^{4}\#\overline{\mathbb{CP}^{2}};\partial S;\alpha)$;
denote it also as $S$ by abuse of notation. Then, the cobordism map
(\ref{eq:khrcobord}) is the image of $S\in{\cal S}_{0}^{2}(D^{4}\#\overline{\mathbb{CP}^{2}};\partial S;\alpha)$
under the following composition
\[
{\cal S}_{0}^{2}(D^{4}\#\overline{\mathbb{CP}^{2}};\partial S;\alpha)\xrightarrow[\cong]{\kappa}{\rm Hom}(KhR_{2}(m(\partial S)),\mathbb{Q})\otimes{\cal S}_{0}^{2}(\overline{\mathbb{CP}^{2}};\alpha)\xrightarrow{{\rm Id}\otimes\theta_{\alpha}}{\rm Hom}(KhR_{2}(m(\partial S)),\mathbb{Q}),
\]
where the first map $\kappa$ is the K\"{u}nneth map for the skein
lasagna module \cite[Theorem 1.4]{MR4445546}, and the second map
is given by evaluating with the canonical dual lasagna generator $\theta_{\alpha}\in({\cal S}_{0}^{2}(\overline{\mathbb{CP}^{2}};\alpha;\mathbb{Q}))_{0,-|\alpha|}^{\ast}$.
\begin{proof}[A skein lasagna proof of Lemma~\ref{lem:tqft-diff} over $\mathbb{Q}$]
We work with the $\mathfrak{gl}_{2}$ homology $KhR_{2}$ and oriented
surfaces $S,S'\subset D^{4}\#\overline{\mathbb{CP}^{2}}$ such that
$\alpha:=[S]=[S']\in H_{2}(D^{4}\#\overline{\mathbb{CP}^{2}})=H_{2}(\overline{\mathbb{CP}^{2}})$.
We show that if there exists a diffeomorphism rel $\partial$
\[
F:(D^{4}\#\overline{\mathbb{CP}^{2}},S)\to(D^{4}\#\overline{\mathbb{CP}^{2}},S'),
\]
then $KhR_{2,\overline{\mathbb{CP}^{2}}}(S')=\pm KhR_{2,\overline{\mathbb{CP}^{2}}}(S)$.

Without loss of generality, we may assume that $F$ fixes a collar
neighborhood of the boundary; assume that it fixes a neighborhood
of the $D^{4}$ summand. Extend $F$ by the identity to a diffeomorphism
$G:\overline{\mathbb{CP}^{2}}\to\overline{\mathbb{CP}^{2}}$. Since
$F$ is the identity on a neighborhood of the $D^{4}$ summand, one
can check by going through the proof of the K\"{u}nneth formula \cite[Theorem 1.4]{MR4445546}
that the following commutes, where $\mathcal{S}_{0}^{2}(F),\mathcal{S}_{0}^{2}(G)$
are the induced isomorphisms on the $\mathfrak{gl}_{2}$ skein lasagna
modules. 
\[\begin{tikzcd}
	{\mathcal{S}_{0}^{2}(D^{4}\#\overline{\mathbb{CP}^{2}};\partial S;\alpha)} & {{\rm Hom}(KhR_{2}(m(\partial S)),\mathbb{Q})\otimes\mathcal{S}_{0}^{2}(\overline{\mathbb{CP}^{2}};\alpha)} \\
	{\mathcal{S}_{0}^{2}(D^{4}\#\overline{\mathbb{CP}^{2}};\partial S';\alpha)} & {{\rm Hom}(KhR_{2}(m(\partial S')),\mathbb{Q})\otimes\mathcal{S}_{0}^{2}(\overline{\mathbb{CP}^{2}};\alpha)}
	\arrow["{\kappa }", from=1-1, to=1-2]
	\arrow["\cong"', from=1-1, to=1-2]
	\arrow["{\mathcal{S}_0 ^2 (F)}", from=1-1, to=2-1]
	\arrow["{{\rm{Id}}\otimes \mathcal{S}_0 ^2 (G)}", from=1-2, to=2-2]
	\arrow["{\kappa }", from=2-1, to=2-2]
	\arrow["\cong"', from=2-1, to=2-2]
\end{tikzcd}\]

Since ${\cal S}_{0}^{2}(G)$ is an isomorphism that preserves the
bigrading of ${\cal S}_{0}^{2}(\overline{\mathbb{CP}^{2}};\alpha)$
and ${\cal S}_{0,0,|\alpha|}^{2}(\overline{\mathbb{CP}^{2}};\alpha;\mathbb{Z})=\mathbb{Z}$,
${\cal S}_{0}^{2}(G)$ fixes the canonical dual lasagna generator
$\theta_{\alpha}\in({\cal S}_{0}^{2}(\overline{\mathbb{CP}^{2}};\alpha))_{0,-|\alpha|}^{\ast}$
up to sign. Hence, 
\begin{align*}
KhR_{2,\overline{\mathbb{CP}^{2}}}(F(S)) & =({\rm Id}\otimes\theta_{\alpha})\circ\kappa(F(S))\\
 & =({\rm Id}\otimes\theta_{\alpha})\circ\kappa\circ{\cal S}_{0}^{2}(F)(S)\\
 & =({\rm Id}\otimes(\theta_{\alpha}\circ{\cal S}_{0}^{2}(G)))\circ\kappa(S)\\
 & =\pm({\rm Id}\otimes\theta_{\alpha})\circ\kappa(S)\\
 & =\pm KhR_{2,\overline{\mathbb{CP}^{2}}}(S).
\end{align*}
\end{proof}

\section{\label{sec:Exotic-compact,-contractible}Exotic Mazur manifolds from
Khovanov homology}

In this section, we use the Khovanov map for oriented cobordisms in
$\mathbb{CP}^{2}$ from Section~\ref{sec:Khovanov-homology-and}
to show Theorem~\ref{thm:Khovanov-homology-can}. We thank Lisa Piccirillo
for suggesting this. Note that we mainly work with the mirrors $m(\Sigma_{k}),m(\Sigma_{k}')\subset D^{4}$
because of our orientation conventions (see Remark~\ref{rem:Our-orientation-conventions,}
and Theorem~\ref{thm:HS1}).
\begin{proof}[Proof of Theorem~\ref{thm:Khovanov-homology-can}]
Let $S_{k}$ (resp.\ $S_{k}'$) be the disk in $D^{4}\#\mathbb{CP}^{2}$
given by \emph{blowing up} the disk $m(\Sigma_{k})\subset D^{4}$
(resp.\ $m(\Sigma_{k}')$) at a point on the disk. In other words,
choose a small closed $4$-ball $B$ such that $(B,B\cap m(\Sigma_{k}))\cong(D^{4},D^{2})$.
In particular, $\partial(B\cap m(\Sigma_{k}))\subset\partial B\cong S^{3}$
is an unknot. Replace $(B,B\cap m(\Sigma_{k}))$ by $(\overline{N}(\mathbb{CP}^{1}),F)$,
where $\overline{N}(\mathbb{CP}^{1})$ is a closed tubular neighborhood
of the core $\mathbb{CP}^{1}$ in $\mathbb{CP}^{2}$, $F$ is a fiber
disk, and $\partial F\subset\partial\overline{N}(\mathbb{CP}^{1})\cong S^{3}$
is identified with $\partial(B\cap m(\Sigma_{k}))\subset\partial B$.
Let $S_{k}$ be the resulting properly embedded surface in $(\mathbb{CP}^{2})^{\circ}:=D^{4}\#\mathbb{CP}^{2}$.
Similarly, let $S_{k}'$ be the surface in $(\mathbb{CP}^{2})^{\circ}$
obtained from $m(\Sigma_{k}')$ analogously.

Then, by Remark~\ref{rem:blowdown}, the maps 
\[
Kh_{\mathbb{CP}^{2}}(S_{k}),Kh_{\mathbb{CP}^{2}}(S_{k}'):Kh(J_{k})\to\mathbb{Z}
\]
are the same as 
\[
Kh(m(\Sigma_{k})),Kh(m(\Sigma_{k}')):Kh(J_{k})\to\mathbb{Z},
\]
respectively.

Hence, $((\mathbb{CP}^{2})^{\circ},S_{k})$ and $((\mathbb{CP}^{2})^{\circ},S_{k}')$
are not diffeomorphic rel $\partial$ by Theorem~\ref{thm:HS1} and
Lemma~\ref{lem:tqft-diff}, and so $(\mathbb{CP}^{2})^{\circ}\backslash N(S_{k})$
and $(\mathbb{CP}^{2})^{\circ}\backslash N(S_{k}')$ are not diffeomorphic
rel $\partial$. Using the same method as the proof of Corollary~\ref{cor:For-,-the},
one can show that their boundary has trivial mapping class group for
all $k\ge1$. Note that their boundary is $S_{-1}^{3}(m(J_{k}))\cong-S_{1}^{3}(J_{k})$.

The $4$-manifold $(\mathbb{CP}^{2})^{\circ}\backslash N(S_{k})$
can be obtained from $D^{4}\backslash N(m(\Sigma_{k}))$ by attaching
a $1$-framed $2$-handle along a meridian of $m(\Sigma_{k})$, and
an analogous statement holds for $(\mathbb{CP}^{2})^{\circ}\backslash N(S_{k}')$.
Hence, the first diagrams of Figure~\ref{fig:handlecalculus} are
$-((\mathbb{CP}^{2})^{\circ}\backslash N(S_{k}))$ and $-((\mathbb{CP}^{2})^{\circ}\backslash N(S_{k}'))$,
and these manifolds are homeomorphic since $D^{4}\backslash N(\Sigma_{k})$
and $D^{4}\backslash N(\Sigma_{k}')$ are homeomorphic rel $\partial$
(Corollary~\ref{cor:For-,-the}). Finally, Figure~\ref{fig:handlecalculus}
checks that they can be simplified to the handle diagrams of Figure~\ref{fig:mazur}.
This proves Theorem~\ref{thm:Khovanov-homology-can}.
\end{proof}
\begin{figure}[p]
\centering{}\includegraphics[scale=0.5]{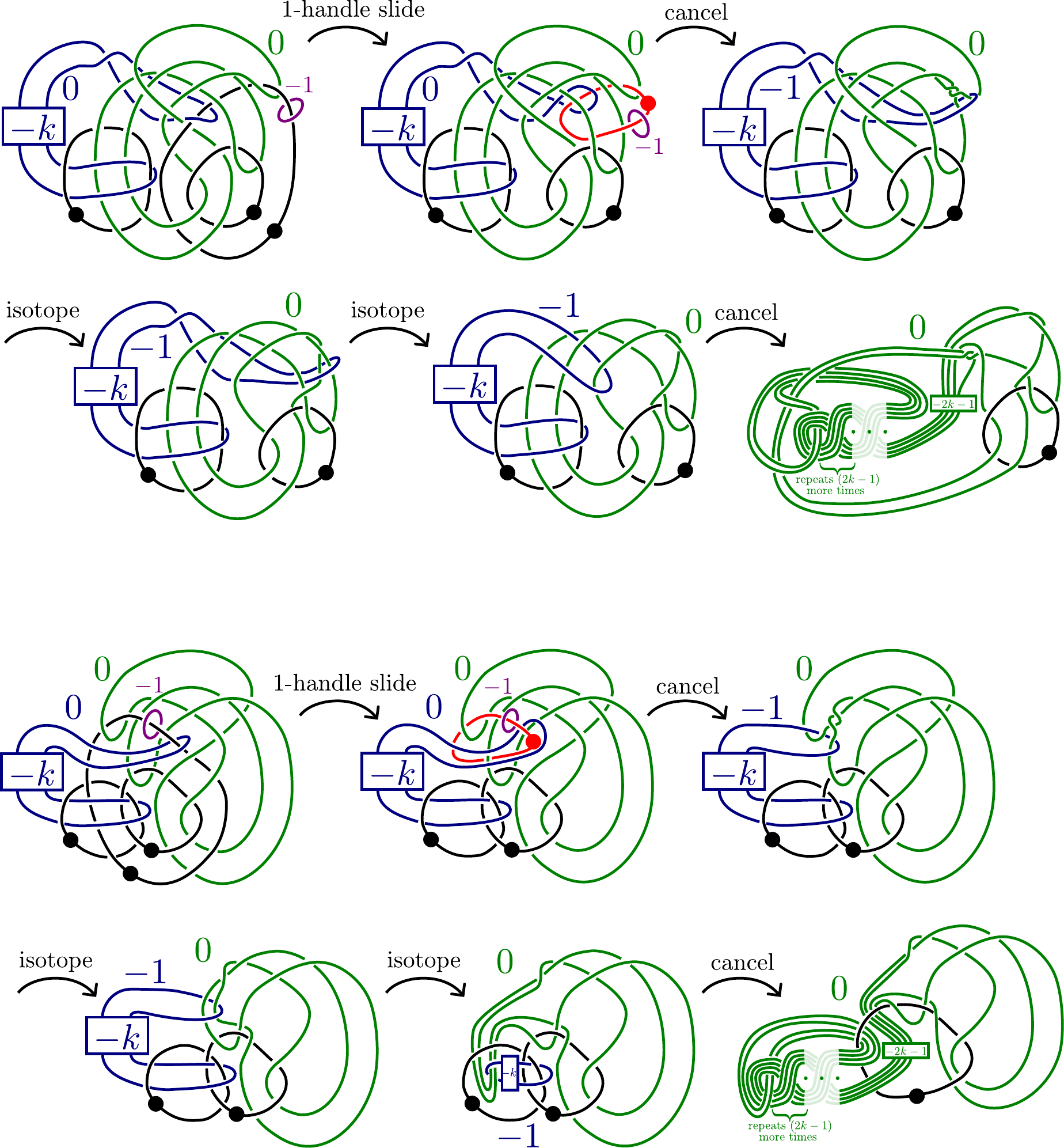}\caption{\label{fig:handlecalculus}Handle calculus that shows that the two
manifolds of Figure~\ref{fig:mazur} are diffeomorphic to the disk
exteriors $-((\mathbb{CP}^{2})^{\circ}\backslash N(S_{k}))$ and $-((\mathbb{CP}^{2})^{\circ}\backslash N(S_{k}'))$
from the proof of Theorem~\ref{thm:Khovanov-homology-can}. In the
final step, the handle diagrams for $k=1$ are shown.}
\end{figure}

\appendix

\section{\label{sec:Computation-of-the}Computation of the mapping class group}

In this appendix, we show that ${\rm MCG}(S_{0}^{3}(J_{k}))$ is trivial
for every integer $k\ge1$, similarly to \cite[Proposition A.3 (a)]{MR4726569},
by using SnapPy \cite{SnapPy} inside Sage \cite{sagemath}. The same
argument works for ${\rm MCG}(S_{1}^{3}(J_{k}))$. Consider the $2$-component
link $J_{0}\sqcup U$ of Figure~\ref{fig:twocomp}. Performing a
$0$-surgery along $J_{0}$ and a $1/k$-surgery along $U$ gives
$S_{0}^{3}(J_{k})$. By first inputting $J_{0}\sqcup U$ \cite[\code{j0u.lnk}]{ancillaryfiles}
as a \code{snappy.ManifoldHP} object called \code{L} and using the
following code, one can verify that $S_{0}^{3}(J_{0})\backslash U$
is hyperbolic and has trivial symmetry group.

\begin{lstlisting}[language=Python,basicstyle={\ttfamily},breaklines=true,showstringspaces=false,columns=fullflexible,keepspaces=true]
L.dehn_fill((0,1),0)
R = L.filled_triangulation().canonical_retriangulation(verified=True)
R.verify_hyperbolicity() # True
len(R.isomorphisms_to(R)) # 1
\end{lstlisting}
Now, Thurston's hyperbolic Dehn surgery theorem \cite{MR4554426}
implies that, for sufficiently large $k$, $S_{0}^{3}(J_{k})$ is
hyperbolic and that the core of the solid torus corresponding to $U$
is the unique shortest geodesic. Hence, any isometry of $S_{0}^{3}(J_{k})$
must fix this geodesic, and so restricts to a diffeomorphism of its
complement, $S_{0}^{3}(J_{0})\backslash U$. Thus we have shown that
${\rm MCG}(S_{0}^{3}(J_{k}))$ is trivial for sufficiently large $k$.
\begin{rem}
\label{rem:Thurston's-theorem-also}Thurston's theorem also says that
${\rm vol}(S_{0}^{3}(J_{k}))\to{\rm vol}(S_{0}^{3}(J_{0})\backslash U)$
as $k\to\infty$, and that if $S_{0}^{3}(J_{k})$ is hyperbolic, then
${\rm vol}(S_{0}^{3}(J_{k}))<{\rm vol}(S_{0}^{3}(J_{0})\backslash U)$.
Hence, there are infinitely many pairwise non-diffeomorphic $S_{0}^{3}(J_{k})$'s
for $k\ge1$. The same conclusion holds for $S_{1}^{3}(J_{k})$, since
one can check that $S_{1}^{3}(J_{0})\backslash U$ is hyperbolic.
\end{rem}

\begin{figure}[H]
\centering{}\includegraphics[scale=0.7]{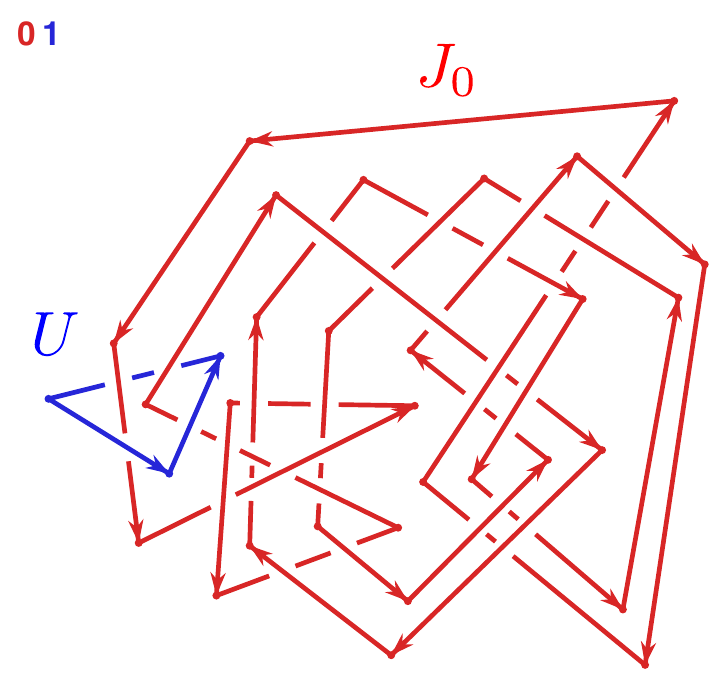}\caption{\label{fig:twocomp}\cite[Figure 14]{MR4726569} A $2$-component
link $J_{0}\sqcup U$ \cite[\code{j0u.lnk}]{ancillaryfiles}.}
\end{figure}

We use the effective bound given by \cite[Theorem 7.28]{MR4466646}
to handle all $k\ge1$. Using the following code, we can check that
$S_{0}^{3}(J_{0})\backslash U$ has no geodesics with length $\le0.1428$,
and hence it satisfies the condition of \cite[Theorem 7.28]{MR4466646},
by \cite[Lemma 7.26]{MR4466646}.
\begin{lstlisting}[language=Python,basicstyle={\ttfamily},breaklines=true,showstringspaces=false,columns=fullflexible,keepspaces=true]
R.length_spectrum_alt(max_len=1, verified=True, bits_prec=1000)
\end{lstlisting}
Now, \code{R.cusp\_areas(verified=True)} outputs $13.24\cdots$,
and hence we use the following code to find all the slopes on the
cusp with length at most $37.5$, which ensures a normalized length
of $>37.5/\sqrt{13.3}>10.1$.
\begin{lstlisting}[language=Python,basicstyle={\ttfamily},breaklines=true,showstringspaces=false,columns=fullflexible,keepspaces=true]
R.short_slopes(verified=True, length=37.5)
\end{lstlisting}
This leaves us to check $k=1,\cdots,17$ in the same way as we checked
$k=1$ in Section~\ref{sec:Exotic-compact,-orientable}. One can
use the following code:
\begin{lstlisting}[language=Python,basicstyle={\ttfamily},breaklines=true,showstringspaces=false,columns=fullflexible,keepspaces=true]
L.dehn_fill((0,1),0)
for k in range(1,18):
  L.dehn_fill((1,k),1)
  print(L.verify_hyperbolicity()[0])
  sp = L.length_spectrum_alt(max_len=1, verified=True, bits_prec=1000)
  print(sp)
  R = L.drill_word(sp[0].word, verified=True, bits_prec=1000).filled_triangulation().canonical_retriangulation(verified=True)
  print("k =",k,"; MCG =",len(R.isomorphisms_to(R)))
\end{lstlisting}

Since \code{verified=True}, if the above code runs successfully,
then SnapPy will have proved, using interval arithmetic, that $S_{0,1/k}^{3}(J_{0},U)$
is hyperbolic and will have provably computed the rank of its mapping
class group. However, this does not mean that $S_{0,1/k}^{3}(J_{0},U)$
being hyperbolic and it having a geodesic of length $\le1$ always
imply that the code will run successfully. Indeed, using the above
given link diagram of $J_{0}\sqcup U$ \cite[\code{j0u.lnk}]{ancillaryfiles},
SnapPy successfully proved that ${\rm MCG}(S_{0,1/k}^{3}(J_{0},U))$
is trivial for all $k\ge1$ and that ${\rm MCG}(S_{1,1/k}^{3}(J_{0},U))$
is trivial for all $k\ge2$, but failed to prove it for $S_{1,1}^{3}(J_{0},U)$.
For this case, working directly with the knot diagram of $J_{1}$
\cite[\code{j1.lnk}]{ancillaryfiles} as in Section~\ref{sec:Exotic-compact,-orientable}
worked for us.

\bibliographystyle{amsalpha}
\bibliography{/Users/gheehyun/Documents/writings/bib}

\end{document}